\documentclass[12pt,reqno]{amsart}
\usepackage[a4paper,twoside,left=2cm,right=2cm,top=3.1cm,bottom=2.3cm]{geometry}

\usepackage[english]{babel}
\usepackage[latin1]{inputenc}       
\usepackage{textcomp}
\usepackage{times}
\usepackage{colortbl}
\usepackage{xcolor}
\usepackage{amsmath, amssymb, amsthm}
\usepackage{mathtools}
\usepackage{mathrsfs}
\usepackage{bbm}
\usepackage{cases}

\usepackage{graphicx} 
\usepackage{svg}
\usepackage{booktabs}
\usepackage{siunitx}
\usepackage{float} 
\usepackage{subfigure} 

\usepackage{url}
\usepackage{enumitem}
\usepackage{color}
\usepackage{epstopdf}

\usepackage{hyperref} 

\usepackage[alphabetic]{amsrefs}

\newtheorem{theorem}{Theorem}[section]
\newtheorem{lemma}[theorem]{Lemma}
\newtheorem{corollary}[theorem]{Corollary}
\newtheorem{proposition}[theorem]{Proposition}

\theoremstyle{definition}
\newtheorem{definition}[theorem]{Definition}

\theoremstyle{remark}
\newtheorem{remark}[theorem]{Remark}

\numberwithin{equation}{section}

\newcommand{\R}{\mathbb{R}}
\newcommand{\N}{\mathbb{N}}
\newcommand{\OU}{\mathscr{L}}

\newcommand{\diam}{\mathrm{diam}}
\newcommand{\tr}{\mathrm{Tr}}
\newcommand{\dd}{\,d}
\newcommand{\dist}{\mathrm{dist}}

\newcommand{\Sub}{J^{2,-}}
\newcommand{\bracket}[1]{\left( #1\right)}

\newcommand{\av}[1]{\left\vert #1\right\vert}
\newcommand{\aV}[1]{\left\Vert #1\right\Vert}

\begin{document}

\title{To $1/2$-logconcavity and beyond: Geometric properties of Dirichlet eigenfunctions}

\author{Lei Qin}
\address{Lei Qin, Institute of Mathematics,
	Hunan University, Changsha, 410082, China}
\email{qlhnumath@hnu.edu.cn}

\author{Jin Sun}
\address{Jin Sun, School of Mathematical Sciences, Fudan University, 200433, Shanghai, China}
\email{jsun22@m.fudan.edu.cn}

\author{Kui Wang}
\address{Kui Wang, School of Mathematical Sciences, Soochow University, 215006, Suzhou, China}
\email{kuiwang@suda.edu.cn}


\begin{abstract}
We prove that, on a bounded open convex domain $\Omega\subset\mathbb{R}^n$, the first Dirichlet eigenfunction of the Laplacian or the Ornstein--Uhlenbeck operator is $\alpha$-logconcave for every $\alpha\in(0,1/2]$. This extends the recent $1/2$-logconcavity theorem of Crasta--Fragal\`{a} for the Laplacian to the weighted Gaussian setting and, simultaneously, to a broader range of exponents. More precisely, if $u$ denotes the first eigenfunction normalized by $\|u\|_\infty=1$, then for every $\alpha\in(0,1/2]$, the function $-\bigl(-\log(\kappa u(x))\bigr)^{\alpha}$ is concave in $\Omega$ provided the scaling parameter $\kappa$ lies below an explicit threshold $\kappa_\alpha(\Omega)\in(0,1)$, which depends on the first Dirichlet eigenvalue and on the diameter of~$\Omega$. For the Ornstein--Uhlenbeck operator, $\kappa_\alpha(\Omega)$ also depends on the distance between $\Omega$ and the origin. Moreover, we establish a local counterpart: for every $\kappa\in(0,1)$, the function $\bigl(-\log(\kappa u)\bigr)^{\alpha}$ is convex on a convex neighborhood $\Omega_\kappa$ of the unique maximum point of~$u$. We also provide counterexamples showing that unscaled $1/2$-logconcavity may fail for the first Dirichlet eigenfunction of a Schr\"odinger operator with a smooth convex potential, and for the first Dirichlet eigenfunction of a weighted Laplacian associated with an affine log-concave weight.
\end{abstract}


\maketitle

\baselineskip18pt

\parskip3pt

\section{Introduction}\label{sec:intro}
The study of logconcavity of ground states for elliptic operators on convex
domains is a classical subject, tracing back at least to the celebrated result by
Brascamp--Lieb~\cite{BrascampLieb1976} for the Laplacian and Schr\"odinger operators with a convex
potential.  In this paper, we focus on
two basic problems: the standard Dirichlet eigenvalue problem for the
Laplacian and the corresponding problem for the Ornstein--Uhlenbeck
operator.

The Dirichlet eigenvalue problem for the Laplacian reads
\begin{equation}\label{eq:evp_1}
  -\Delta u = \lambda_0\,u \quad\text{in }\Omega,
  \qquad u=0 \quad\text{on }\partial\Omega,
\end{equation}
where $\lambda_0=\lambda_0(\Omega)>0$ denotes the first eigenvalue.

For the Ornstein--Uhlenbeck operator
\begin{equation}\label{eq:OU-def}
  \mathcal{L} u := \Delta u - x\cdot\nabla u,
  \qquad x\in\R^n,
\end{equation}
defined on a bounded open convex domain $\Omega\subset\R^n$, the
associated Dirichlet eigenvalue problem reads
\begin{equation}\label{eq:evp_2}
  -\mathcal{L} u = \lambda_1\,u \quad\text{in }\Omega,
  \qquad u=0 \quad\text{on }\partial\Omega,
\end{equation}
with $\lambda_1=\lambda_1(\Omega)>0$ denoting the first eigenvalue.

Both eigenfunctions are commonly referred to as the
\emph{ground state} of the operator with Dirichlet boundary conditions.
For a unified treatment, we shall investigate~\eqref{eq:evp_1}
and~\eqref{eq:evp_2} as instances of
\begin{equation}\label{eq:evp}
  -\OU_b u = \lambda_b\,u \quad\text{in }\Omega,
  \qquad u=0 \quad\text{on }\partial\Omega,
\end{equation}
where 
\begin{equation*}
    \OU_b u := \Delta u - b\,x\cdot\nabla u
\end{equation*}
with parameter
$b\in\{0,1\}$ (the Laplacian case corresponds to $b=0$ and the
Ornstein--Uhlenbeck case to $b=1$).  Throughout, $\lambda_b>0$ denotes the
first Dirichlet eigenvalue of~$-\OU_b$ on~$\Omega$ and~$u$ is the
corresponding eigenfunction, which may be chosen strictly positive
in~$\Omega$ and normalized so that $\max_{\overline\Omega}u=1$.

The classical result of Brascamp--Lieb~\cite{BrascampLieb1976} ensures
that the ground state of the Laplacian on a bounded convex domain is
log-concave. Analogous conclusions are known to hold in the drifted
setting~\cites{ColesantiFranciniLivshytsSalani2024, SunWang2025, Qin2026}, 
for the Schr\"{o}dinger operator \cites{Colesanti2025,Carbotti2026} 
and in many other settings \cites{Sakaguchi1987,CrastaFragala2021}.
This property is also deeply connected with the Brunn--Minkowski type inequalities (see, e.g., \cites{Colesanti2005,CrastaFragala2020}) and the fundamental gap (see, e.g., \cites{AndrewsClutterbuck2011, ni2013, SetoWangWei2019}). 
More recently, Ishige--Salani--Takatsu \cites{IshigeSalaniTakatsu2020,IshigeSalaniTakatsu2022,IshigeSalaniTakatsu2024} have systematically developed the framework of generalized $F$-concavity, investigating whether the ground states satisfy \emph{stronger} concavity conditions.

Following the framework of~\cite{IshigeSalaniTakatsu2020}, we introduce the
family of \emph{$\alpha$-log transformations}
\[
  L_\alpha(s):=-(-\log s)^\alpha,
  \qquad \alpha\in(0,1],\quad s\in(0,1].
\]
\begin{definition}
A bounded non-negative function $u$ with
$\max_{\Omega}u=1$ is said to be \emph{scaled $\alpha$-logconcave with
constant $\kappa_0>0$} in~$\Omega$ if for every $\kappa\in(0,\kappa_0]$
the function $L_\alpha(\kappa u)$ is concave in~$\Omega$, i.e.,
    \begin{equation}\label{eq:alpha-logconc}
        L_\alpha\!\bracket{\kappa u((1-t)x+ty)}
        \ge
        (1-t)\,L_\alpha(\kappa u(x))
        +t\,L_\alpha(\kappa u(y))
    \end{equation}
holds for all $x,y\in\Omega$ and $t\in[0,1]$.  When $\kappa_0=1$,
    we say that $u$ is \emph{unscaled $\alpha$-logconcave} in~$\Omega$.
\end{definition}

For $\alpha=1$, this is precisely the usual notion of logconcavity, and
scaled logconcavity trivially implies logconcavity.  For every
$\beta\ge\alpha$, the function $f(t):=-(-t)^{\beta/\alpha}$ is increasing
and concave on $(-\infty,0]$, so $L_\beta(\kappa u)=f\circ
L_\alpha(\kappa u)$ inherits concavity from $L_\alpha(\kappa u)$.
Consequently, $\alpha$-logconcavity implies $\beta$-logconcavity, and
hence, for $0<\alpha<1$, $\alpha$-logconcavity is a strictly stronger
property than logconcavity.

To motivate the threshold $\alpha=1/2$, recall the two model kernels
to which the present problems are related.
For the heat kernel (centered at the origin)
\begin{align*}
	H(x,t):=\frac{1}{\bracket{4\pi t}^{{n}/{2}}}\exp\bracket{-\frac{\av{x}^2}{4t}},
\end{align*}
and for any fixed $t>0$ and every sufficiently small $\kappa>0$, the function
\begin{equation*}
	-(-\log (\kappa H(x,t)))^{1/2} = -\bracket{\frac{\av{x}^2}{4t}-\log\bracket{\bracket{4\pi t}^{-{n}/{2}}\kappa}}^{1/2}
\end{equation*}
is concave in $\R^n$.

The Mehler kernel, which serves as the fundamental solution for the
Ornstein--Uhlenbeck heat equation, is given by
\begin{align*}
    M(x,y,t):=\frac{1}{\bracket{1-e^{-2t}}^{{n}/{2}}}\exp\bracket{-\frac{1}{2}\frac{e^{-2t}(\av{x}^2+\av{y}^2)-2e^{-t}x\cdot y}{1-e^{-2t}}}.
\end{align*}
Evaluating the Mehler kernel at $y=0$, we obtain the explicit formula
\begin{align*}
	M(x,t):=\frac{1}{\bracket{1-e^{-2t}}^{{n}/{2}}}\exp\bracket{-\frac{1}{2}\frac{e^{-2t}\av{x}^2}{1-e^{-2t}}}.
\end{align*}
For any fixed $t>0$ and every sufficiently small $\kappa>0$,
\begin{equation*}
	-(-\log (\kappa M(x,t)))^{1/2} = -\bracket{\frac{1}{2}\frac{e^{-2t}\av{x}^2}{1-e^{-2t}}-\log\bracket{\bracket{1-e^{-2t}}^{-{n}/{2}}\kappa}}^{1/2}
\end{equation*}
is also concave in $\R^n$.  Both kernels are therefore scaled
$\alpha$-logconcave precisely for $1/2\le\alpha\le 1$, and the exponent
$\alpha=1/2$ is critical.  It is consequently natural to ask whether the
ground state of the Laplacian or the Ornstein--Uhlenbeck operator on
a bounded convex domain inherits this scaled $\alpha$-logconcavity, and
in particular whether this property is preserved down to $\alpha=1/2$ and beyond.

For the Laplacian, Crasta and Fragal\`{a}~\cite{CrastaFragala2026}
recently proved that the ground state on a bounded convex
$\Omega\subset\R^n$ is scaled $1/2$-logconcave, thereby resolving a
question that the heat-flow approach
of~\cite{IshigeSalaniTakatsu2020} had left open.  Our principal aim is
to extend their result in two directions.  First, we admit the
Ornstein--Uhlenbeck operator alongside the standard Laplacian.  Second,
we obtain scaled $\alpha$-logconcavity for the entire subcritical range
$0<\alpha\le 1/2$.

Our first main result asserts that, after rescaling by a small
factor $\kappa$, the ground state is genuinely $\alpha$-logconcave on the
whole convex domain.  The explicit threshold
$\kappa_\alpha(\Omega)$ below depends on the geometry and
on the eigenvalue. The Laplacian case for $\alpha=1/2$ is proved in~\cite{CrastaFragala2026}*{Theorem 1}.

\begin{theorem}[Global $\alpha$-logconcavity]\label{thm:global}
Let $0<\alpha\le 1/2$. Let $\Omega\subset\R^n$ be open, bounded, and convex, and let $u$ be the first Dirichlet eigenfunction to the problem \eqref{eq:evp}, normalized so that $\max_{\overline\Omega}u=1$.  Write
$\lambda_b=\lambda_b(\Omega)$ for the corresponding
eigenvalue and $\overline\lambda_b$ for the first Dirichlet
eigenvalue to the operator $-\frac{d^2}{ds^2} + bs\cdot \frac{d}{ds} $ on the interval $(-D_\Omega/2,D_\Omega/2)$, where
$D_\Omega=\diam(\Omega)$.  Define $\kappa_\alpha(\Omega)$ by
\begin{equation}\label{eq:threshold}
  -\log\kappa_\alpha (\Omega) := \max\left\{\frac{2^{\beta}(\beta-1)}{(2^{\beta}-1)\beta}\,,\frac{\overline\lambda_b D_\Omega^2}{2^{\beta-2}\beta(\beta-1)}\,,\frac{(\beta-1)2^{2\beta-1}\lambda_b}{\beta\,\overline{\lambda}_b}\,,\frac{(\beta-1)2^{3\beta-2}b\max_{\Omega}\av{x}^2}{\beta\overline{\lambda}_b}\right\},
\end{equation}
where $\beta=1/\alpha$.
Then for every $\kappa\in(0,\kappa_\alpha(\Omega)]$,
\begin{equation}\label{eq:global-concavity}
  L_{\alpha}\!\bracket{\kappa u((1-t)x+ty)}
  \ge
  (1-t)L_{\alpha}(\kappa u(x))
  +tL_{\alpha}(\kappa u(y))
\end{equation}
holds for all $x,y\in\Omega$ and $t\in[0,1]$.
\end{theorem}
For the Laplacian case, the threshold $\kappa_\alpha(\Omega)$ depends on $\alpha$, the first
Dirichlet eigenvalues $\lambda_b, \overline\lambda_b$, and the diameter $D_\Omega$ of~$\Omega$.
For the Ornstein-Uhlenbeck case, the threshold $\kappa_\alpha(\Omega)$ additionally depends on $\max_{\Omega}\av{x}$.

Our second result is local and unconditional in $\kappa\in(0,1)$:
the ground state is scaled $\alpha$-logconcave in a convex neighborhood of its
unique maximum point, and this neighborhood shrinks to the unique maximum point of~$u$ as $\kappa$ approaches $1$.  In particular, the local statement
captures the geometry that is invisible in the global one
near the critical value $\kappa=1$. The Laplacian case for $\alpha=1/2$ is proved in~\cite{CrastaFragala2026}*{Theorem 3}.

\begin{theorem}[Local $\alpha$-logconcavity]\label{thm:local}
Let $0<\alpha\le 1/2$. Let $\Omega\subset\R^n$ be open, bounded, and convex, and let $u$ be the first Dirichlet eigenfunction to the problem \eqref{eq:evp}, normalized so that $\max_{\overline\Omega}u=1$. 
For every $\kappa\in(0,1)$, set $w_\kappa:=(-\log(\kappa u))^{\alpha}$. Then there exists
$\overline u_\kappa\in(0,1)$ with
$\lim_{\kappa\to 1^-}\overline u_\kappa=1$, such that
$w_\kappa$ is convex on the convex sublevel set
\[
  \Omega_\kappa
  :=
  \bigl\{x\in\Omega:\, u(x)>\overline u_\kappa\bigr\}.
\]
In particular, \eqref{eq:global-concavity} holds for all $x,y\in\Omega_\kappa$ and $t\in[0,1]$.
\end{theorem}

Unlike the heat kernel, it is remarkable that the ground state on a bounded convex domain is also scaled $\alpha$-logconcave for every $\alpha\in(0,1/2]$. The Dirichlet eigenvalue problem
on a bounded domain differs from the heat kernel owing to the boundary condition.
Scaling by $\kappa<1$ replaces $v=-\log u$ by $a+v$ with $a=-\log\kappa>0$.
Near a maximum point one has $v=0$ and $\nabla v=0$, so the positive
term $(a+v)D^2v$ in~\eqref{eq:Hesscriterion} dominates locally and
forces $u$ to be $\alpha$-logconcave in a convex neighborhood of its
maximum point.  Globally, in contrast to the Gaussian decay of $H(x,t)$ or $M(x,t)$ on $\R^n$, the vanishing of $u$ on $\partial\Omega$ yields a
stronger log-concavity.

Without scaling, no such subcritical behaviour can be expected.  If
$\kappa=1$ and $x_0$ is the unique maximum of~$u$, then along a unit
vector~$e$,
\[
        v(x_0+te)=\frac12D^2v(x_0)[e,e]\,t^2+O(t^3),
\]
which implies
\[
        v(x_0+te)^\alpha=C_e|t|^{2\alpha}(1+o(1)),
        \qquad C_e>0.
\]
For $0<\alpha<1/2$, the one-dimensional function $|t|^{2\alpha}$ is
concave on each side of the origin rather than convex.  This is the reason why the scaling parameter $\kappa<1$ is unavoidable in Theorem~\ref{thm:global} and Theorem~\ref{thm:local}.

\medskip\noindent\textbf{Strategy of the proof.}\;
Our argument for Theorem~\ref{thm:global} adapts the convex envelope method
of~\cite{CrastaFragala2026} to both the Laplacian and the
Ornstein--Uhlenbeck operator and to the full subcritical range
$0<\alpha\le 1/2$. We transform the eigenvalue equation into a PDE for $w_\kappa=(-\log(\kappa u))^{\alpha}$ and aims to show that $w_\kappa$ coincides with its convex envelope $w_\kappa^{**}$.
Three ingredients are needed.

First, the improved logconcavity estimate of Andrews and
Clutterbuck~\cite{AndrewsClutterbuck2011} for the Laplacian, together
with its extension by Sun and Wang~\cite{SunWang2025} for the
Ornstein--Uhlenbeck ground state, provides a lower bound on the
gradient of $w_\kappa$ at the support points of the convex envelope
(Proposition~\ref{prop:grad-est}). Second, combining this
gradient bound with a careful analysis of the nonlinearity, we verify
that $w_\kappa^{**}$ is a viscosity supersolution to the same PDE as
$w_\kappa$, provided $\kappa\le\kappa_\alpha(\Omega)$
(Proposition~\ref{prop:supersol}). 
Third, a variational argument exploiting the simplicity of the first
eigenvalue yields $w_\kappa=w_\kappa^{**}$, completing the proof of Theorem~\ref{thm:global}
(Section~\ref{sec:conclusion}).

\medskip\noindent\textbf{Outline.}\; Section~\ref{sec:prelim} collects
the necessary preliminary lemmas.  Section~\ref{sec:main-proof}
contains the proof of Theorem~\ref{thm:global}.  In
Section~\ref{sec:local} we establish Theorem~\ref{thm:local}.  Finally,
Section~\ref{sec:counter} presents counterexamples showing that
$1/2$-logconcavity fails for the first Dirichlet eigenfunction of
general Schr\"odinger operators with smooth convex potentials, and
for weighted Laplacians with affine log-concave weights.

\section{Preliminaries}\label{sec:prelim}

Throughout this paper, $\Omega\subset\R^n$ is an open, bounded, and convex
domain, and $u\in C^\infty(\Omega)$ denotes the first Dirichlet
eigenfunction to the problem \eqref{eq:evp} , normalized so that
$\max_{\overline\Omega}u=1$, with eigenvalue $\lambda_b>0$.

\subsection{Gradient estimates}\label{subsec:grad-est}

The gradient estimate for the ground state will be used in the proof of Theorem~\ref{thm:global}. It is a standard consequence of elliptic
regularity: since $u\in C^2(\Omega)\cap C(\overline\Omega)$ and $\Omega$
is bounded and convex, the gradient $\nabla u$ is bounded on~$\Omega$.  Explicitly, there exists a constant $C= C(n,\Omega)>0$ such that
\begin{equation}\label{eq:grad_u_boudned}
    \aV{\nabla u}_{L^\infty(\Omega)}\le C.
\end{equation}

For the sake of completeness, we provide the proof for the Ornstein-Uhlenbeck case (the Laplacian case follows similarly).
\begin{theorem}
Let $\Omega \subset \mathbb{R}^n$ be a bounded convex domain with diameter $D_\Omega$. Suppose that $u \in C^2(\Omega) \cap C(\overline{\Omega})$ is the first Dirichlet eigenfunction of the operator $\mathcal{L} = \Delta - x \cdot \nabla$, satisfying $-\mathcal{L} u = \lambda_1 u$ in $\Omega$, $\max_{\overline\Omega}u=1$ and $u = 0$ on $\partial\Omega$. Then there exists a constant $C > 0$, depending only on $n$, $\lambda_1$, $D_\Omega$, and $\max_{\Omega}|x|$, such that $\|\nabla u\|_{L^\infty(\Omega)} \le C$.
\end{theorem}

\begin{proof}
Let $K = \max_{\Omega}|x|$. Since $\Omega$ is a bounded convex domain, for any $P \in \partial\Omega$, there exists a supporting hyperplane $H_P$ such that $\Omega$ lies on one side of $H_P$. Let $n_P$ denote the inner unit normal to $H_P$. We define the distance from $x$ to $H_P$ as $d_P(x) = (x - P) \cdot n_P$. Note that $0 < d_P(x) \le D_\Omega$ for all $x \in \Omega$.

Consider the barrier function $\varphi_P(x) = A(1 - e^{-\gamma d_P(x)})$, where $A, \gamma > 0$ are constants to be determined. A direct computation yield
\[
\mathcal{L} \varphi_P = -A e^{-\gamma d_P} (\gamma^2 + \gamma x \cdot n_P).
\]
Since $x \cdot n_P \ge -|x| \ge -K$ and $e^{-\gamma d_P} \ge e^{-\gamma D_\Omega}$, it follows that
\[
\mathcal{L} \varphi_P \le -A e^{-\gamma D_\Omega} (\gamma^2 - \gamma K).
\]
We choose $\gamma = K + 1$ so that $\gamma^2 - \gamma K = \gamma > 0$. To ensure $-\mathcal{L} \varphi_P \ge \lambda_1$, we set $A = \lambda_1 \gamma^{-1} e^{\gamma D_\Omega}$. Consequently, $-\mathcal{L} \varphi_P \ge \lambda_1 \ge \lambda_1 u = -\mathcal{L} u$ in $\Omega$. On $\partial\Omega$, since $u = 0$ and $\varphi_P \ge 0$, we have $u \le \varphi_P$ on $\partial\Omega$. By the weak maximum principle for elliptic operators, we have $u(x) \le \varphi_P(x)$ for all $x \in \Omega$. Therefore,
\[
u(x) \le A(1 - e^{-\gamma d_P(x)}) \le A\gamma d_P(x).
\]
Taking the infimum over all $P \in \partial\Omega$ and using the fact that $\inf_{P \in \partial\Omega} d_P(x) = \text{dist}(x, \partial\Omega) := \delta(x)$, we establish an upper bound
\[
u(x) \le C_0 \delta(x), \quad \forall x \in \Omega,
\]
where $C_0 = A\gamma = \lambda_1 e^{(K+1)D_\Omega}$.

For any fixed point $x_0 \in \Omega$, set $r = \frac{1}{2}\delta(x_0)$. By standard interior elliptic estimates (see, e.g., \cite{GilbargTrudinger2001}*{Corollary 6.3}), there exists a constant $C_1 > 0$ depending on $n$, $\lambda_1$, $D_\Omega$, and $K$, such that
\[
\|\nabla u\|_{L^\infty(B_{r/2}(x_0))} \le \frac{C_1}{r} \|u\|_{L^\infty(B_r(x_0))}.
\]
For any point $x = x_0 + ry \in B_r(x_0)$, the triangle inequality yields $\delta(x) \le \delta(x_0) + r = 3r$. Hence, $u(x) \le C_0 \delta(x) \le 3C_0 r$. Therefore, we have
\[
 |\nabla u(x_0)| \le \frac{C_1}{r} \|u\|_{L^\infty(B_r(x_0))}\le 3 C_0 C_1.
\]
Since $3 C_0 C_1$ is independent of the choice of $x_0$, the gradient $\nabla u$ is uniformly bounded in $\Omega$, which completes the proof.
\end{proof}

\subsection{Improved logconcavity}\label{subsec:improved}

The key ingredient in the proof of Theorem~\ref{thm:global} is the following improved logconcavity inequality.  Let $\overline\varphi$ and $ \overline{\lambda}_b$ denote the first Dirichlet eigenfunction and eigenvalue to the problem \eqref{eq:evp} on the interval $(-D_\Omega/2,D_\Omega/2)$, respectively,
normalized so that $\max\overline\varphi=1$, and set
$\overline v:=-\log\overline\varphi$.

\begin{lemma}[{Improved logconcavity, \cites{AndrewsClutterbuck2011,SunWang2025}}]
\label{lem:improved-logconc}
Set $v:=-\log u$.  Then, for every pair of distinct points $z,y\in\Omega$,
\begin{equation}\label{eq:improved-logconc}
  \Bigl\langle
    \nabla v(z)-\nabla v(y),\;
    \frac{z-y}{|z-y|}
  \Bigr\rangle
  \ge
  2\,\overline v'\!\bracket{\frac{|z-y|}{2}}.
\end{equation}
\end{lemma}
\begin{remark}
In the Laplacian case ($b=0$), one has $\overline
v'(s)=\dfrac{\pi}{D_\Omega}\tan\!\bracket{\dfrac{\pi}{D_\Omega}s}$, so
that $\overline v''(s)\ge\dfrac{\pi^2}{D_\Omega^2}=\overline\lambda_0$
and $\overline v'(s)\ge s\,\overline\lambda_0$ for
$s\in[0,D_\Omega/2)$.  In the Ornstein--Uhlenbeck case ($b=1$), the
function~$\overline\varphi$ is logconcave and even, so $\overline v$ is
convex with $\overline v'(0)=0$; moreover, from the ODE satisfied by
$\overline v$ one deduces $\overline v''\ge\overline\lambda_1$ on
$[0,D_\Omega/2)$, which yields
\begin{equation}\label{eq:Vprime-lower}
  \overline v'(s)\ge s\,\overline\lambda_1
  \qquad\forall\, s\in[0,D_\Omega/2).
\end{equation}
Therefore, in both cases, we have $\overline
v''(s)\ge\overline\lambda_b$ and $\overline v'(s)\ge
s\,\overline\lambda_b$ on $[0,D_\Omega/2)$, and
Lemma~\ref{lem:improved-logconc} yields the pointwise Hessian estimate
\begin{equation}\label{eq:Hess-lower}
    \nabla^2 v\ge \overline\lambda_b\,I_n>0
    \qquad\text{in }\ \Omega,
\end{equation}
which implies $u$ is strongly logconcave.
\end{remark}

\subsection{Convex envelope}\label{subsec:conv-env}

For a lower semicontinuous function $w:\Omega\to\R$ such that $w(x)\to+\infty$ as $\dist(x,\partial\Omega)\to 0$, its \emph{convex envelope} in~$\Omega$ is given by
\begin{equation}\label{eq:conv-env-def}
  w^{**}(x)
  :=
  \inf\Bigl\{
    \sum_{i=1}^m t_i\,w(x_i):\;
    m\in\N,\; t_i > 0,\; x_i\in\Omega,\;
    \sum_{i=1}^m t_i=1,\; \sum_{i=1}^m t_i x_i=x
  \Bigr\}.
\end{equation}
The points $\{x_i\}_{i=1}^m$ for which this infimum is attained,
i.e.\ such that $w^{**}(x)=\sum_{i=1}^m t_i\,w(x_i)$, are called the
\emph{support points} of $w^{**}$ at~$x$.

Recall that the \emph{second-order subjet} of a function~$w$ at~$x$ is
\begin{align*}
	\Sub w(x)
	:=
	\bigl\{
	(p,\mathcal{A})\in\R^n\!\times\!\R^{n\times n}_{\mathrm{sym}}:&
	\exists\,\varphi\in C^\infty(\R^n)\text{ with }
	w-\varphi\text{ having a local min at }x,\;\\
	&(p,\mathcal{A})=(\nabla\varphi(x),\nabla^2\varphi(x))
	\bigr\}.
\end{align*}

The following lemma is standard (see, e.g., \cites{AlvarezLasryLions1997,CrastaFragala2026} for detailed proofs).

\begin{lemma}\label{lem:convex_envelope}
Let $w\in C^2(\Omega)$ satisfy $w(x)\to+\infty$ as
$\dist(x,\partial\Omega)\to 0$. Then
$w^{**}(x)\to+\infty$ as $\dist(x,\partial\Omega)\to 0$. For each $x\in\Omega$, the infimum in \eqref{eq:conv-env-def} is attained with $m\le n+1$ and
$t_1,\dots,t_m>0$ such that 
\begin{equation}\label{eq:support_points}
	\sum_{i=1}^m t_i x_i=x, \quad w^{**}(x) = \sum_{i=1}^m t_i\,w(x_i).
\end{equation}
Moreover, if $(p,\mathcal{A})\in \Sub w^{**}(x)$, then:
\begin{enumerate}
\item[\emph{(i)}] $w^{**}(x_i) = w(x_i)$ for every $i$, and
\begin{equation}\label{eq:w_affine}
     w(x_i) = w^{**}(x) + p\cdot(x_i-x),\qquad i=1,\ldots,m.
\end{equation}
\item[\emph{(ii)}] $\nabla w(x_i)=p$ for every $i$.
\item[\emph{(iii)}] Each Hessian
  $\mathcal{A}_i:=\nabla^2 w(x_i)$ is positive semidefinite.
\item[\emph{(iv)}] If every $\mathcal{A}_i$ is positive definite, then
\begin{equation*}
    \mathcal{A} \le \bracket{\sum_{i=1}^m t_i\,(\mathcal{A}_i)^{-1}}^{-1}.
\end{equation*}
\item[\emph{(v)}] 
If every $\mathcal{A}_i$ is non-zero, then
\begin{equation}
    \tr(\mathcal{A})
   \le
   \bracket{\sum_{i=1}^m t_i\,\tr(\mathcal{A}_i)^{-1}}^{-1}.
\end{equation}
If $\mathcal{A}_i=\mathbf{0}$ for some $i$, then $\tr(\mathcal{A}) \le 0$.
\end{enumerate}
\end{lemma}
\begin{remark}
    Since $ w^{**} \leq w$ and $w^{**}$ is convex, equality in \eqref{eq:support_points} implies $w^{**}(x_i) = w(x_i)$ and that $w^{**}$ is affine on the convex hull of the support points. Therefore, every slope $p$ at $x_i$ is a slope of this affine restriction, which yields \eqref{eq:w_affine}.
\end{remark}

\section{Global \texorpdfstring{$\alpha$}{alpha}-logconcavity}\label{sec:main-proof}

This section is devoted to the proof of
Theorem~\ref{thm:global}. For fixed $\alpha \in (0,\frac{1}{2}]$ and $\kappa\in(0,1)$, set $v=-\log u$ and define
\begin{equation}\label{eq:vkappa-def}
  w_\kappa
  :=\bracket{-\log(\kappa u)}^{\alpha}
  =\bracket{-\log\kappa + v}^{\alpha}
  \qquad\text{in }\ \Omega.
\end{equation}
Since $0<u\le 1$ and $0<\kappa<1$, we have $\min_\Omega w_\kappa=(-\log\kappa)^{\alpha}>0$ and $w_\kappa(x)\to+\infty$ as $\dist(x,\partial\Omega)\to 0$.

\subsection{The transformed equation}\label{subsec:transformed}

Set $\beta = \alpha^{-1}\ge 2$. A direct computation starting from~\eqref{eq:evp} shows that $w_\kappa$ is a classical solution to
\begin{equation}\label{eq:w-pde}
  -\Delta w + b\,x\cdot\nabla w
  + \left(\beta w^{\beta-1}-\frac {\beta-1}{w}\right)|\nabla w|^2
        +\frac{\lambda_b}{\beta w^{\beta-1}}
  = 0
  \qquad\text{in }\ \Omega.
\end{equation}
It is convenient to express the left-hand side through the function
\begin{equation}\label{eq:N-def}
  \mathcal{N}_b[s,\xi,\mathcal{X};x]
  :=
  -\tr(\mathcal{X}) + b\,x\cdot\xi
  + \left(\beta s^{\beta-1}-\frac {\beta-1}{s}\right)|\xi|^2
        +\frac{\lambda_b}{\beta s^{\beta-1}},
\end{equation}
where $(s,\xi,\mathcal{X})\in(0,\infty)\times\R^n\times\R^{n\times
n}_{\mathrm{sym}}$ and $x\in\R^n$ denotes the spatial position (arising
from the drift term). The equation~\eqref{eq:w-pde} then takes the form
$\mathcal{N}_b[w_\kappa,\nabla w_\kappa,\nabla^2 w_\kappa\,;x]=0$ in $\Omega$.

Write $w_\kappa^{**}$ for the convex envelope of $w_\kappa$ in $\Omega$, and set
\begin{equation}\label{eq:ukappa-def}
  \hat u_\kappa
  :=
  \exp\!\bracket{-\!\log\kappa-(w_\kappa^{**})^{\beta}},
\end{equation}
so that 
\begin{equation*}
  w_\kappa^{**}
  = \bracket{-\log(\kappa\,\hat u_\kappa)}^{\alpha}.
\end{equation*}

The proof  of Theorem~\ref{thm:global} proceeds through three propositions.

\subsection{Gradient bound at support points}
\label{subsec:grad-bound}

Our first technical step is a lower bound on the gradient of
$w_\kappa$ at the support points of its convex envelope.  This bound
will play a decisive role when verifying that $w_\kappa^{**}$ is a
viscosity supersolution of~\eqref{eq:w-pde}. The estimate of the following proposition is new for $\beta>2$.
\begin{proposition}\label{prop:grad-est}
Let $\kappa\in(0,1)$.  Suppose $w_\kappa^{**}(x)<w_\kappa(x)$ at some $x\in\Omega$, and let $x_1,\dots,x_m\in\Omega$ ($m\ge 2$) be support points for $w_\kappa^{**}(x)$.  Write $p_\kappa$ for the common gradient $\nabla w_\kappa(x_i)$ (cf.\ Lemma~\ref{lem:convex_envelope}\emph{(ii)}).  If we choose $M_0:=w_\kappa(x_1)\le w_\kappa(x_2)\le\cdots\le w_\kappa(x_m)=:M_1$ and $\tau = M_1/M_0> 1$, then
\begin{equation}\label{eq:p-lower}
  |p_\kappa|^2 \ge \max\left\{\frac{\overline\lambda_b}{\beta(\beta-1)(\tau M_0)^{\beta-2}},\frac{(\tau-1)^2M_0^2}{D_\Omega^2}\right\}.
\end{equation}
\end{proposition}

\begin{proof}
Setting $v=-\log u$, the chain rule yields
$\nabla v=\beta\,w_\kappa^{\beta-1}\nabla w_\kappa$. Since $w_\kappa^{**}$ is affine on the segment joining any two support points $x_i,x_j$, we obtain
\begin{equation*}
	w_\kappa(x_i)-w_\kappa(x_j)=\langle p_\kappa,\,x_i-x_j\rangle.
\end{equation*}
By substituting $\nabla v(x_k)=\beta\,w_\kappa^{\beta-1}(x_k)\,p_\kappa$ into the improved log-concavity estimate~\eqref{eq:improved-logconc} and applying it to the pair $(x_1,x_m)$, we obtain
\begin{align}
    \frac{\beta(M_1^{\beta-1}-M_0^{\beta-1})(M_1-M_0)}{|x_m-x_1|}
  &=\beta(w_\kappa^{\beta-1}(x_m)-w_\kappa^{\beta-1}(x_1))\Bigl\langle p_\kappa,\,\frac{x_m-x_1}{|x_m-x_1|}\Bigr\rangle\notag\\
  &\ge
  \overline{\lambda}_b |x_m-x_1|.\label{eq:M_1M_0_lower}
\end{align}
Since $\beta \in [2,\infty)$ and $t^{\beta-1}$ is a convex function in $(0,+\infty)$, we have
\[
M_1^{\beta-1}-M_0^{\beta-1} \leq (\beta - 1)M_1^{\beta-2}(M_1-M_0),
\]
which, combined with~\eqref{eq:M_1M_0_lower}, yields
\[
\beta( \beta-1)M_1^{\beta-2}\bracket{\frac{M_1-M_0}{|x_m-x_1|}}^2\ge \overline{\lambda}_b.
\]
Using the elementary inequality
\begin{equation}\label{eq:M_1M_0_upper}
    M_1-M_0=\langle p_\kappa,\,x_m-x_1\rangle\leq \av{p_\kappa}|x_m-x_1|,
\end{equation}
we deduce
\[
\beta(\beta-1)M_1^{\beta-2}\av{p_\kappa}^2\ge \overline{\lambda}_b,
\]
which is the first lower bound in~\eqref{eq:p-lower}. Finally, note that inequality \eqref{eq:M_1M_0_upper} also implies
\[
M_1-M_0\leq \av{p_\kappa}|x_m-x_1|\le D_\Omega\av{p_\kappa},
\]
from which the second lower bound in~\eqref{eq:p-lower} follows immediately.
\end{proof}

\begin{remark}
For $1<\tau\le 2$, the first lower bound in~\eqref{eq:p-lower} reads
 \[
    |p_\kappa|^2 \ge \frac{\overline\lambda_b}{\beta(\beta-1)2^{\beta-2}M_0^{\beta-2}}.
    \]
    For $\tau>2$, if we choose 
    \begin{equation}\label{eq:kappa_1_alpha}
        \kappa\le \exp\bracket{-\frac{\overline\lambda_b D_\Omega^2}{2^{\beta-2}\beta(\beta-1)}}=:\kappa_{\alpha,1},
\end{equation}
then
\[
    M_0\ge (-\log\kappa)^{\alpha}\ge \bracket{\frac{\overline\lambda_b D_\Omega^2}{2^{\beta-2}\beta(\beta-1)}}^{1/\beta},
\]
and the second lower bound \eqref{eq:p-lower} gives
\[
    |p_\kappa|^2 \ge\frac{M_0^2}{D_\Omega^2}\ge \frac{\overline\lambda_b}{\beta(\beta-1)2^{\beta-2}M_0^{\beta-2}}.
\]
    Consequently, for $\kappa\le\kappa_{\alpha,1}$,
    the bound~\eqref{eq:p-lower} simplifies to
    \begin{equation}\label{eq:p-lower-useful}
        |p_\kappa|^2 \ge \frac{\overline\lambda_b}{\beta(\beta-1)2^{\beta-2}M_0^{\beta-2}},
    \end{equation}
    which is more convenient for our proof.
\end{remark}

\subsection{Supersolution property of the convex envelope}
\label{subsec:supersol}

The key argument of the proof of Theorem~\ref{thm:global} consists in verifying that $w_\kappa^{**}$
satisfies $\mathcal{N}_b\ge 0$ in the viscosity sense when $\kappa\le\kappa_\alpha(\Omega)$.

\begin{proposition}\label{prop:supersol}
For every $\kappa\in(0,\kappa_\alpha(\Omega)]$, the convex envelope $w_\kappa^{**}$ is a viscosity supersolution of
equation~\eqref{eq:w-pde}in the sense that, for any fixed $x\in\Omega$, if $(p,\mathcal{A})\in\Sub
w_\kappa^{**}(x)$, then 
\begin{equation}\label{eq:supersol-ineq}
  \mathcal{N}_b[w_\kappa^{**}(x),\,p,\,\mathcal{A}\,;x]
  \ge 0.
\end{equation}
\end{proposition}

\begin{proof}
If $w_\kappa^{**}(x)=w_\kappa(x)$, then $\Sub w_\kappa^{**}(x)\subseteq\Sub w_\kappa(x)$, and
\eqref{eq:supersol-ineq} follows immediately from the fact that $w_\kappa$ solves
\eqref{eq:w-pde} classically. In particular, this case covers the
situation $p=0$, since the unique minimum of $w_\kappa$ coincides with
the unique maximum of~$u$, where $w_\kappa^{**}=w_\kappa$.

Therefore, it remains to consider the case where $w_\kappa^{**}(x)<w_\kappa(x)$ and $p\ne 0$. By Lemma \ref{lem:convex_envelope}, there exist  $t_1,\dots,t_m \in (0,1)$ and $x_1,\dots,x_m\in\Omega$ ($m\ge 2$) such that
\[
\mathop{\sum}_{i=1}^{m} t_{i}=1,\quad x=\mathop{\sum}_{i=1}^{m} t_i x_i,\quad w^{**}_k (x)=\mathop{\sum}_{i=1}^{m} t_i w_\kappa(x_i).
\]
Abbreviate
\[
  w^*:=w_\kappa^{**}(x),
  \qquad
  w_i:=w_\kappa(x_i),\qquad p:=\nabla w_\kappa(x_i),
  \qquad
  \mathcal{A}_i:=\nabla^2 w_\kappa(x_i).
\]
Without loss of generality, assume that $M_0 = w_1\le w_2\le\cdots\le w_m = M_1$ with $\tau = M_1/M_0>1$.

Since $w^{**}$ is affine on the convex hull of the support points, the identity \eqref{eq:w_affine} yields 
\begin{equation}\label{000}
	x_i\cdot p = w_i + (x\cdot p-w^{*})=w_i+\eta
\end{equation}
for some constant $\eta$ independent of $i$.  For any $\beta \ge 2$, we define
\[
G(s) := b(s+\eta)
  + \left(\beta s^{\beta-1}-\frac {\beta-1}{s}\right)|p|^2
        +\frac{\lambda_b}{\beta s^{\beta-1}}.
\]
Since $w_\kappa$ satisfies~\eqref{eq:w-pde} at $x_i$ classically, we have
\begin{align*}
    \mathcal{N}_b[w_i,p,\mathcal{A}_i;x_i]
    & = -\tr(\mathcal{A}_i) + b\,x_i\cdot p
  + \left(\beta w_i^{\beta-1}-\frac{\beta-1}{w_i}\right)|p|^2
        +\frac{\lambda_b}{\beta w_i^{\beta-1}}\\
        &=-\tr(\mathcal{A}_i) + b(w_i+\eta)
  + \left(\beta w_i^{\beta-1}-\frac{\beta-1}{w_i}\right)|p|^2
        +\frac{\lambda_b}{\beta w_i^{\beta-1}}\\
         &=-\tr(\mathcal{A}_i) + G(w_i)\\
         &=0,
\end{align*}
which implies
\begin{equation}\label{eq:trace_A_i_G}
    \tr(\mathcal{A}_i) = G(w_i).
\end{equation}
 For $M_0\le s\le M_1$, by identity \eqref{000}, we have
\[
\av{s+\eta}\le \max\{\av{M_0+\eta}, \av{M_1+\eta}\}\le \max_{i}\av{x_i}\av{p}\le \frac{\beta}{2^{\beta+1}} s^{\beta-1}|p|^2 + \frac{2^{\beta-1}\max_{\Omega}\av{x}^2}{\beta s^{\beta-1}},
\]
which yields 
\begin{equation}\label{eq:lower_G_s}
    G(s)> \frac{(2^{\beta+1}-1)\beta}{2^{\beta+1}}s^{\beta-1}|p|^2-\frac{2^{\beta-1}b\max_{\Omega}\av{x}^2}{\beta s^{\beta-1}}-\frac {\beta-1}{s}|p|^2.
\end{equation}
Choose
\begin{equation}\label{eq:kappa_2_alpha}
    \kappa_{\alpha,2}:=\min\left\{\exp\bracket{-\frac{2^{\beta}(\beta-1)}{(2^{\beta}-1)\beta}}, \exp\bracket{-\frac{(\beta-1)2^{3\beta-2}b\max_{\Omega}\av{x}^2}{\beta\overline{\lambda}_b}}\right\}.
\end{equation}
Then from the gradient bound \eqref{eq:p-lower-useful}, we have 
\begin{equation}\label{eq:kappa_2_G}
    \frac{\beta}{2^{\beta+1}}s^{\beta-1}|p|^2\ge\frac{2^{\beta-1}b\max_{\Omega}\av{x}^2}{\beta s^{\beta-1}},
    \qquad
\frac{(2^{\beta}-1)\beta}{2^{\beta}}s^{\beta-1}|p|^2\ge \frac {\beta-1}{s}|p|^2
\end{equation}
for $\kappa\le\kappa_{\alpha,2}$ and $s\ge M_0\ge (-\log\kappa)^\alpha$. Hence,
\[
G(s)> 0
\]
for every $s\ge M_0$.

By Lemma \ref{lem:convex_envelope}, we know that all the matrices $\mathcal{A}_i$ are positive semi-definite. To establish \eqref{eq:supersol-ineq}, we distinguish two cases.

\noindent\textsc{Case 1: } every $\mathcal{A}_i$ is non-zero.
By Lemma~\ref{lem:convex_envelope} {(v)} and the
identity~\eqref{eq:trace_A_i_G},
\begin{align*}\label{eq:F-lower1}
  \mathcal{N}_b[w^*,p,\mathcal{A};x]
  &= -\tr(\mathcal{A}) + b\,x\cdot p
  + \left(\beta (w^{*})^{\beta-1}-\frac{\beta-1}{w^{*}}\right)|p|^2
        +\frac{\lambda_b}{\beta (w^{*})^{\beta-1}}\\
    &=-\tr(\mathcal{A}) +G(w^{*})\ge
  -\bracket{\sum_{i=1}^m t_i
    G(w_i)^{-1}}^{\!-1}
  + G(w^*),
 \end{align*}
It therefore suffices to show
\begin{equation}\label{eq:target}
  G(w^*)^{-1} \le \sum_{i=1}^m t_i
    G(w_i)^{-1},
\end{equation}
from which~\eqref{eq:supersol-ineq} follows at once. By Jensen's
inequality, \eqref{eq:target} holds whenever $G(s)^{-1}$ is convex on $[M_0,M_1]$. 

Differentiating $G(s)$ twice yields
\begin{align*}
    G'(s) &= b
  + \left(\beta(\beta-1) s^{\beta-2}+\frac {\beta-1}{s^2}\right)|p|^2
        -\frac{(\beta-1)\lambda_b}{\beta s^{\beta}},\\
        G''(s)&=\left(\beta(\beta-1)(\beta-2) s^{\beta-3}-\frac {2(\beta-1)}{s^3}\right)|p|^2
        +\frac{(\beta-1)\lambda_b}{s^{\beta+1}}.
\end{align*}
Since 
\[
\bracket{G(s)^{-1}}'' = \frac{2G'(s)^2-G''(s)G(s)}{G(s)^3},
\]
for convexity of $G^{-1}$, it suffices to show
\begin{align}\label{eq:target2}
    2G'(s)^2-G''(s)G(s)\ge 0.
\end{align}
Choose 
\begin{equation}\label{eq:kappa_alpha_3}
    \kappa_{\alpha,3}:=\exp\bracket{-\frac{(\beta-1)2^{2\beta-1}\lambda_b}{\beta\,\overline{\lambda}_b}}.
\end{equation}
Then from the gradient bound \eqref{eq:p-lower-useful}, we have 
\[
2^{-\beta-1}\beta(\beta-1) s^{\beta-2}|p|^2\ge\frac{(\beta-1)\lambda_b}{\beta s^{\beta}},
\qquad
2^{-\beta-1}\beta s^{\beta-1}|p|^2\ge\frac{\lambda_b}{\beta s^{\beta-1}}
\]
for $\kappa\le\kappa_{\alpha,3}$ and $s\ge M_0\ge (-\log\kappa)^\alpha$. Combining the above inequalities with the inequality \eqref{eq:kappa_2_G}, we obtain
\begin{align}
    G(s) &\le 
    \frac{(2^{\beta+1}+1)\beta}{2^{\beta+1}}s^{\beta-1}|p|^2
    +\frac{2^{\beta-1}\max_{\Omega}b\av{x}^2}{\beta s^{\beta-1}}
    +\frac{\lambda_b}{\beta s^{\beta-1}}\notag\\
    &\le \frac{(2^{\beta+1}+2)\beta}{2^{\beta+1}}s^{\beta-1}|p|^2
    +\frac{2^{\beta-1}\max_{\Omega}b\av{x}^2}{\beta s^{\beta-1}}\notag\\
&\le \frac{(2^{\beta-1}+1)\beta}{2^{\beta-1}}s^{\beta-1}|p|^2,
    \notag
\end{align}
and
\begin{align}
    G'(s) &\ge 
    b
    +\beta(\beta-1) s^{\beta-2}|p|^2
        -\frac{(\beta-1)\lambda_b}{\beta s^{\beta}}\notag\\
    &\ge b+\frac{(2^{\beta+1}-1)}{2^{\beta+1}}\beta(\beta-1)s^{\beta-2}|p|^2.\notag
\end{align}
Similarly, we have
\begin{align}
    G''(s) &\le 
    \beta(\beta-1)(\beta-2) s^{\beta-3}|p|^2
        +\frac{(\beta-1)\lambda_b}{s^{\beta+1}}\notag\\
    &\le\beta(\beta-1)\bracket{\beta+2^{-\beta-1}\beta-2} s^{\beta-3}|p|^2
    \notag\\
    &\le\beta(\beta-1)^2 s^{\beta-3}|p|^2.
    \notag
\end{align}
For $G''\le 0$, it is easy to see that \eqref{eq:target2} holds. So we may assume $G''(s)>0$.
Therefore, 
\begin{align*}
    2G'(s)^2-G''(s)G(s)
&\ge 2\bracket{b+\frac{(2^{\beta+1}-1)}{2^{\beta+1}}\beta(\beta-1)s^{\beta-2}|p|^2}^2  - \frac{(2^{\beta-1}+1)}{2^{\beta-1}}\beta^2(\beta-1)^2s^{2\beta-4}|p|^4\\
&\ge  \bracket{\frac{(2^{\beta+1}-1)^2}{2^{2\beta+1}}-\frac{(2^{\beta-1}+1)}{2^{\beta-1}}}\beta^2(\beta-1)^2s^{2\beta-4}|p|^4\\
&\ge \frac{1}{2^{2\beta+1}}\beta^2(\beta-1)^2s^{2\beta-4}|p|^4.
\end{align*}
Hence, \eqref{eq:target2} holds, i.e., $G^{-1}(s)$ is convex on $[M_0,M_1]$ for $\kappa\le \kappa_\alpha(\Omega)$ with
\begin{align*}
    \kappa_\alpha(\Omega) := \min\left\{\kappa_{\alpha,1},\kappa_{\alpha,2},\kappa_{\alpha,3}\right\}.
\end{align*}
Together with \eqref{eq:kappa_1_alpha},  \eqref{eq:kappa_2_alpha}, and  \eqref{eq:kappa_alpha_3}, this yields the definition \eqref{eq:threshold} for $\kappa_\alpha(\Omega)$.
Thus, whenever $\kappa\le \kappa_\alpha(\Omega)$, $\mathcal{N}_b[w_\kappa^{**}(x),\,p,\,\mathcal{A}\,;x]
  \ge 0$ for any $x\in\Omega$ and $(p,\mathcal{A})\in\Sub
w_\kappa^{**}(x)$.


\medskip
\noindent\textsc{Case 2: }some $\mathcal{A}_i=\mathbf{0}$.\;
By Lemma~\ref{lem:convex_envelope} {(v)}, $\tr(\mathcal{A})\le 0$.
Since $w^*\ge M_0$, it follows that
\[
  \mathcal{N}_b[w_\kappa^{**},p,\mathcal{A};x]
  \ge
  G(w^*)>0
\]
for any $x\in\Omega$ and $(p,\mathcal{A})\in\Sub
w_\kappa^{**}(x)$.

Therefore, for every $\kappa\in(0,\kappa_\alpha(\Omega)]$, the convex envelope $w_\kappa^{**}$ is a viscosity supersolution of
equation~\eqref{eq:w-pde}.
\end{proof}

\subsection{Completion of the proof}\label{sec:conclusion}

We now translate the supersolution property of $w_\kappa^{**}$
back to a subsolution property of the corresponding candidate
eigenfunction $\hat u_\kappa$, that is,
\[
\hat u_{\kappa}=\exp \left(-\log \kappa-(w_\kappa^{**})^{1/\alpha})\right)\quad \text{i.e.,}\quad w_\kappa ^{**}=(-\log (\kappa\hat  u_\kappa))^{\alpha},\quad \forall \alpha\in \left(0, 1/2\right].
\]

Following the approach in~\cite{CrastaFragala2026}, we establish the following proposition, which completes the proof of Theorem~\ref{thm:global}.
\begin{proposition}\label{prop:subsol}
Let $\kappa\in(0,\kappa_\alpha(\Omega)]$ and let $\hat u_\kappa$ be
defined by~\eqref{eq:ukappa-def}.  Then:
\begin{enumerate}
\item[\emph{(i)}]
  $\hat u_\kappa$ is a viscosity and distributional subsolution to
  $-\OU_b\hat u_\kappa=\lambda_b\,\hat u_\kappa$ in~$\Omega$;
\item[\emph{(ii)}]
  $-\OU_b\hat u_\kappa\le\lambda_b\,\hat u_\kappa$ holds a.e.\ in
  $\Omega$;
  \item[\emph{(iii)}]
  $\max_{\overline\Omega}\hat u_\kappa=1$,
  $\hat u_\kappa\in W^{1,2}_{0}(\Omega)$ for $b=0$ and  $\hat u_\kappa\in W^{1,2}_{0}(\Omega,\gamma)$ for $b=1$.
\end{enumerate}
\end{proposition}

\begin{proof}
(i) The statement is a direct translation of the supersolution property
of~$w_\kappa^{**}$ (Proposition~\ref{prop:supersol}) through the
change of variables~\eqref{eq:ukappa-def}, since a smooth function~$\varphi$
touching $\hat u_\kappa$ from above at~$x$ corresponds, via $\psi=(-\log(\kappa\varphi))^{\alpha}$, to a smooth function touching
$w_{\kappa}^{**}$ from below at~$x$.  Furthermore, the operator is linear and uniformly elliptic with continuous coefficients, so by \cite{Ishii1995}*{Theorem 1 and Theorem 2}, viscosity and distributional subsolutions agree.

(ii) The statement follows from~(i) together with Alexandrov's theorem,
which guarantees that the convex function $w_\kappa^{**}$ (hence
$\hat u_\kappa$) is twice differentiable almost everywhere.

(iii) First, note that the identity $\max\hat u_\kappa=1$ is a consequence of
$\min w_\kappa^{**}=\min w_\kappa=(-\log\kappa)^{\alpha}$. 

To establish the Lipschitz bound, we observe that for every
$x\in\Omega$ there exists $y_x\in\Omega$ with
$w_\kappa(y_x)\le w_\kappa^{**}(x)$ and
$\nabla w_\kappa^{**}(x)=\nabla w_\kappa(y_x)$
(cf.\ Lemma~\ref{lem:convex_envelope}).  
Define $\Phi_\kappa(s):=s^{-1+1/\alpha}\exp(-\log\kappa-s^{1/\alpha})$. Observe that $\Phi_\kappa(s)$ is monotone decreasing
on $[(1-\alpha)^{\alpha},+\infty)$. 
From \eqref{eq:threshold}, we obtain
\begin{equation*}
    \kappa\le\kappa_\alpha(\Omega)\le e^{-\frac{2^{\beta}(\beta-1)}{(2^{\beta}-1)\beta}}<e^{\alpha-1}.
\end{equation*}
This ensures that $\Phi_\kappa(s)$ is monotone decreasing on $[(-\log\kappa)^{\alpha},\infty)$. Thus,
\[
  |\nabla\hat u_\kappa(x)|
  = (1/\alpha)\,\Phi_\kappa(w_\kappa^{**}(x))\,|\nabla w_\kappa^{**}(x)|
  \le (1/\alpha)\,\Phi_\kappa(w_\kappa(y_x))\,|\nabla w_\kappa(y_x)|
  = |\nabla u(y_x)|.
\]
Consequently, $|\nabla\hat u_\kappa|$ is bounded since $\|\nabla u\|_{L^\infty(\Omega)}<\infty$.

Therefore, the Lipschitz regularity of $\hat u_\kappa$ and the fact that $\hat u_\kappa(x)\to 0$ as $x\to\partial\Omega$ (by Lemma~\ref{lem:convex_envelope}) ensure  $\hat u_\kappa\in W^{1,2}_{0}(\Omega)$ for $b=0$ and  $\hat u_\kappa\in W^{1,2}_{0}(\Omega,\gamma)$ for $b=1$.
\end{proof}

\begin{proof}[Proof of Theorem~\ref{thm:global}]
We present the argument in the Ornstein--Uhlenbeck case; the Laplacian case is entirely analogous. Multiplying the inequality in
Proposition~\ref{prop:subsol}(ii) by $\hat u_\kappa$ and integrating
over~$\Omega$ with respect to the Gaussian weight
$d\gamma=e^{-|x|^2/2}\,dx$, we obtain
\[
  \int_\Omega |\nabla\hat u_\kappa|^2\,d\gamma
  \;\le\;
  \lambda_1
  \int_\Omega |\hat u_\kappa|^2\,d\gamma.
\]
Since $\hat u_\kappa\in W^{1,2}_{0}(\Omega;\gamma)\setminus\{0\}$ and
$\lambda_1$ is the infimum of the Rayleigh quotient over this
space, the inequality above can hold only if $\hat u_\kappa$ is a
first eigenfunction.  By the simplicity of the first eigenvalue, there
exists $c\in\R$ such that $\hat u_\kappa=cu$ in~$\Omega$.
Comparing the maxima shows $c=1$, hence $\hat u_\kappa=u$, which is
equivalent to $w_\kappa^{**}=w_\kappa$.  This is exactly the
$\alpha$-logconcavity statement~\eqref{eq:global-concavity}.
\end{proof}

\section{Local \texorpdfstring{$\alpha$}{alpha}-logconcavity}\label{sec:local}

In this section, we prove Theorem~\ref{thm:local}.  The
argument is essentially a pointwise computation.
Fix $\kappa\in(0,1)$. Let $v=-\log u$, $a=-\log\kappa>0$, and
\[
        w_\kappa= (-\log(\kappa u))^\alpha  = (a+v)^\alpha,\quad \forall \alpha\in (0,1].
\]
A direct computation yields
\begin{align}\label{eq:Hessianidentity}
        \nabla w_\kappa&=\alpha(a+v)^{\alpha-1}\nabla v,
        \notag\\
        \nabla^2w_\kappa&=\alpha(a+v)^{\alpha-2}
        \bigl((a+v)\nabla^2v-(1-\alpha)\nabla v\otimes\nabla v\bigr).
\end{align}
Consequently, $L_{\alpha}(\kappa u)=-w_\kappa$ is concave if and only
if~$w_\kappa$ is convex.  Pointwise convexity of~$w_\kappa$ is in turn equivalent to
\begin{equation}\label{eq:Hesscriterion}
        (a+v)\nabla^2 v\ge (1-\alpha)\nabla v\otimes\nabla v
\end{equation}
in the sense of quadratic forms. 

We are now ready to prove Theorem~\ref{thm:local} via the improved logconcavity estimate~\eqref{eq:Hess-lower}.

\begin{proof}[Proof of Theorem~\ref{thm:local}]
Combining \eqref{eq:Hess-lower} with \eqref{eq:Hessianidentity} yields
\begin{equation}\label{eq:Hessian_w_lowerbound}
    \nabla^2w_\kappa(\xi,\xi)
        \ge \alpha(a+v)^{\alpha-2}
        \left(\overline{\lambda}_b(a+v)-(1-\alpha)|\nabla v|^2\right)|\xi|^2,\quad \forall \xi \in \R^n.
\end{equation}
The bound $D^2v\ge \overline{\lambda}_bI_n$ implies that $v$ is strongly convex, so $v$ has a unique minimizer $x_0$.  Since $u$ is normalized by $\max_\Omega u=1$, we have $v(x_0)=0$, and $\nabla v(x_0)=0$. 

Define 
\[
M(t) = \sup\{|\nabla v(x)|^2:x\in\Omega,\; v(x)\leq t\}.
\]
Then $M(t)$ is monotonic increasing in $[0,\infty)$. The sublevel sets $\{x: v(x)\leq t\}$ shrink to $\{x_0\}$ as $t\to 0^+$, and continuity of $\nabla v$ gives $M(t)\to0$ as $t\rightarrow 0^{+}$. Moreover, since $\nabla v = \nabla u/u$ and $u(x)\to 0$ as $\dist(x,\partial\Omega)\to 0$, we have $M(t)\to+\infty$ as $t\to +\infty$. Since $M(t)$ is monotonic increasing and continuous, for every $\kappa\in(0,1)$, there exists a constant $\overline{v}_\kappa$ such that
\[
M(\overline{v}_\kappa) = \frac{a\overline{\lambda}_b}{2(1-\alpha)}.
\]
Since $a=-\log\kappa$ and~$M$ is monotone increasing,
$\overline{v}_\kappa$ may be chosen monotone increasing in~$a$, and
$\overline{v}_\kappa\to 0^+$ as $\kappa\to 1^-$.
Let $\overline{u}_\kappa = e^{-\overline{v}_\kappa}$. Then $\lim_{\kappa\to 1^-}\overline u_\kappa=1$ and since $v$ is strongly convex, the sublevel set
\[
\Omega_\kappa:=\{x\in\Omega:\, u(x)>\overline u_\kappa\} = \{x\in\Omega:\, v(x)<\overline v_\kappa\}
\]
is convex.
Furthermore, for any $x\in\Omega_\kappa$, the inequality \eqref{eq:Hessian_w_lowerbound} gives
\[
        \nabla^2w_\kappa(\xi,\xi)
        \ge \frac{a\alpha\overline{\lambda}_b}{2}(a+v)^{\alpha-2}
        |\xi|^2>0.
\]
Therefore, $w_\kappa=(-\log(\kappa u))^{\alpha}$ is strongly convex on the convex sublevel set $\Omega_\kappa$, which is the desired conclusion.
\end{proof}

\section{The first Dirichlet eigenfunction fails to be unscaled \texorpdfstring{$1/2$}{1/2}-logconcave for general Schr\"odinger operators and for general weighted Laplacian operators}\label{sec:counter}

In this section, we focus on the $1/2$-logconcavity of the first Dirichlet eigenfunction $\psi$ of the Schr\"odinger operator in one dimension. We establish that for both the Laplace and Ornstein-Uhlenbeck operators, the first Dirichlet eigenfunction is unscaled $1/2$-logconcave; for a one-dimensional Schr\"odinger operator with a convex potential, we prove a scaled positive result.
Nevertheless, we exhibit two natural classes of examples in which the
unscaled $1/2$-logconcavity property fails: a Schr\"odinger operator
with the convex potential $V(x)=|x|$, and a weighted Laplacian
associated with the affine log-concave weight $e^{-cx}$. These counterexamples show that the threshold $\kappa<1$ in Theorem~\ref{thm:global} cannot be replaced by $\kappa=1$ for general operators.

Throughout this section, $\psi$ denotes the first Dirichlet
eigenfunction of the operator under consideration, and we set $v:=-\log\psi$.
From \eqref{eq:Hesscriterion}, $w = \sqrt{a+v}$ is convex if and only if
\begin{equation*}
        2(a+v)\nabla^2 v\ge \nabla v\otimes\nabla v
\end{equation*}
In one dimension, this reduces to
\begin{equation}\label{eq:half-criterion-1d}
        2(a+v)v''-(v')^2\ge0.
\end{equation}

\subsection{First Dirichlet eigenfunctions are unscaled \texorpdfstring{$1/2$}{1/2}-logconcave}\label{subsec:1d-unscaled}
We begin by noting that, for both the one-dimensional Laplacian operator
and the one-dimensional Ornstein--Uhlenbeck operator, the first
Dirichlet eigenfunction is \emph{unscaled} $1/2$-logconcave.  This
positive result provides a useful contrast to the counterexamples that
follow.
\begin{proposition}\label{prop:1d-unscaled_ground_state}
     Let $\psi_b$ be the first Dirichlet function with $\psi_b(0)=1$ and let $\overline\lambda_b$ be the first Dirichlet eigenvalue to the operator $-\frac{d^2}{ds^2} + bs\cdot \frac{d}{ds} $ on the interval $(-L/2,L/2)$. Then $\psi_b\le 1$ and $\psi_b$ is unscaled $1/2$-logconcave, that is, $\sqrt{-\log\psi_b}$ is convex in $(-L/2,L/2)$.
\end{proposition}
\begin{proof}
    For the Laplacian case, $b=0$ and $\psi_0 = \cos({\pi s}/{L})\le 1$. Setting $v = -\log\psi_0$, a direct computation gives
    \[
    v'(s) = \frac{\pi}{L}\frac{\sin({\pi s}/{L})}{\cos({\pi s}/{L})},\qquad
    v''(s) = \frac{\pi^2}{L^2}\frac{1}{\cos^2({\pi s}/{L})}.
    \]
    Define $Q(s) := 2v(s)v''(s)-(v'(s))^2$. Differentiating $Q(s)$ yields
    \[
    Q'(s)  =  2v(s)v'''(s) = 4\frac{\pi^3}{L^3}\frac{\sin({\pi s}/{L})}{\cos^3({\pi s}/{L})}v(s)>0
    \]
    in $(0,L/2)$. Hence, $Q(s)>Q(0) = 0$ in $(0,L/2)$. Since $Q(s)$ is an even function, we have  $Q(s)\ge 0$ in $(-L/2,L/2)$. By criterion~\eqref{eq:half-criterion-1d}, $\sqrt{-\log\psi_b}$ is therefore convex in $(-L/2,L/2)$.
    
    For the Ornstein-Uhlenbeck case ($b=1$), we first claim $\psi_1''<0$ in $(-L/2,L/2)$. Indeed, suppose for contradiction that there exists a point $s_0\in(-L/2,L/2)$ at which $\psi_1''(s_0)=0$. Since $\psi_1$ is even and $\psi_1''(0)=-\lambda_1<0$, we may assume that $s_0>0$ is the smallest zero of $\psi_1''$ in $(0,L/2)$.  On $(0,s_0]$, the function~$\psi_1'$ is negative and
    \[
    \psi_1''(s)=-\lambda_1\psi_1(s) + s\psi_1'(s)\le-\lambda_1\psi_1(s_0)<0,
    \]
    which contradicts $\psi_1''(s_0)=0$.  Hence, $\psi_1''<0$ on $(-L/2,L/2)$.
    Setting $v = -\log\psi_1$, we have, for $s\in (0,L/2)$,
    \[
        v'(s)=-\frac{\psi^\prime(s)}{\psi(s)}>0,
        \qquad
        v''(s)=sv'(s) + v'(s)^2+\lambda_1>0,
    \]
    and
    \[
    v'''(s) = (s+2v'(s))v''(s) + v'(s)>0.
    \]
    As before, define $Q(s) := 2v(s)v''(s)-(v'(s))^2$. Then
    \[
    Q'(s)  =  2v(s)v'''(s) >0
    \]
    on $(0,L/2)$, so that $Q\ge 0$ on $(-L/2,L/2)$ since $Q$ is even.  By criterion~\eqref{eq:half-criterion-1d}, $\sqrt{-\log\psi_1}$ is convex on $(-L/2,L/2)$.
\end{proof}

\subsection{One-dimensional convex Schr\"odinger potentials: a scaled positive result}\label{subsec:1d-positive}

\begin{proposition}\label{prop:1d-scaled}
Let $I=(c_1,c_2)$ be a bounded interval, and let $V\in C^2(I)\cap C(\overline I)$ be convex.  Let $\psi>0$ be the first Dirichlet eigenfunction of
\[
        -\psi''+V\psi=\lambda\psi,
        \qquad \psi(c_1)=\psi(c_2)=0,
\]
normalized by $\max_I\psi=1$.  Then there exists $a_0<\infty$ such that $\sqrt{a-\log\psi}$ is convex on $I$ for every $a\ge a_0$.
\end{proposition}

\begin{proof}
Let $v=-\log\psi$. The logconcavity comparison theorem \cite{AndrewsClutterbuck2011}*{Theorem 1.5} gives $v''>0$ in $I$.  Since
\[
        v'(x)=-\frac{\psi^\prime(x)}{\psi(x)},
        \qquad
        v''(x)=\frac{\psi^\prime(x)^2-\psi^{\prime\prime}(x)\psi(x)}{\psi(x)^2},
\]
we obtain
\[
\frac{(v')^2}{v''}  = \frac{\psi^\prime(x)^2}{\psi^\prime(x)^2-\psi^{\prime\prime}(x)\psi(x)} \to1
\]
at both endpoints $c_1$ and $c_2$.

On every compact subinterval of $I$, the ratio $(v')^2/v''$ is continuous and finite since $v''>0$.  Consequently,
\[
        M=\sup_I\frac{(v')^2}{v''}<\infty.
\]
Taking $a_0=M/2$, the one-dimensional criterion \eqref{eq:half-criterion-1d} gives
\[
        2(a+v)v''-(v')^2
        \ge 2a v''-(v')^2
        \ge 0
\]
for every $a\ge a_0$, as claimed.
\end{proof}

\subsection{The full-line Airy ground state}\label{subsec:Airy}

Define the quadratic form
\begin{equation}\label{eq:Airy-form}
        q_\infty[u]=\int_\R |u'(x)|^2\dd x+
        \int_\R |x|\,|u(x)|^2\dd x
\end{equation}
on
\begin{equation}\label{eq:Airy-domain}
        D_\infty=\{u\in H^1(\R): |x|^{1/2}u\in L^2(\R)\}.
\end{equation}
Set $H_\infty = -d^2/dx^2+|x|$.

To describe the ground state of~$H_\infty$, we briefly recall the basic properties of Airy functions (see, e.g., \cite{Olver1997}). Let $\operatorname{Ai}(x)$ and $\operatorname{Bi}(x)$ denote Airy functions of the first and second kind, respectively, which are linearly independent solutions to the standard Airy differential equation
\begin{equation*}
    -y''(x) + x y(x) = 0.
\end{equation*}
From their classical asymptotic expansions, it is well known that as $x \to +\infty$, $\operatorname{Ai}(x)$ decays exponentially while $\operatorname{Bi}(x)$ grows exponentially. As $x \to -\infty$, both solutions exhibit oscillatory behavior, and $\operatorname{Ai}(x)$ along with its derivative $\operatorname{Ai}'(x)$ possess infinitely many strictly negative real zeros.

The following theorem collects the basic facts about the ground
state of~$H_\infty$.  Although these are well known, we provide a
self-contained proof for completeness.
\begin{theorem}[Full-line Airy ground state]\label{thm:Airy-full}
The operator $H_\infty$ has compact resolvent.  Its lowest eigenvalue $\lambda_\infty$ is simple, and the positive ground state is unique up to multiplication by a positive constant.  Normalized such that $\psi_\infty(0)=1$, the function is even and
\begin{equation}\label{eq:Airy-ground}
        \psi_\infty(x)=
        \frac{\operatorname{Ai}(|x|-\lambda_\infty)}
             {\operatorname{Ai}(-\lambda_\infty)},
\end{equation}
where $\lambda_\infty$ is the first positive zero of
\begin{equation}\label{eq:Airy-zero}
        \operatorname{Ai}'(-\lambda)=0.
\end{equation}
\end{theorem}

\begin{proof}
The form \eqref{eq:Airy-form} is closed on $D_\infty$ relative to the norm
\[
        \|u\|_{D_\infty}^2=\|u\|_{L^2(\R)}^2+q_\infty[u],
\]
since any Cauchy sequence is Cauchy both in $H^1(\R)$ and in $L^2(\R,|x|\dd x)$, and the two limits agree as distributions. It follows from \cite{Kato1995}*{Chapter  VI} that $q_\infty$
uniquely determines the self-adjoint operator $H_\infty$.

Next, we show that the embedding $D_\infty\hookrightarrow L^2(\R)$ is compact. Indeed, if $(u_n)$ is bounded in $D_\infty$, then for $R>0$,
\[
        \int_{|x|>R}|u_n(x)|^2\dd x
        \le \frac1R\int_{|x|>R}|x|\,|u_n(x)|^2\dd x
        \le\frac CR.
\]
On $[-R,R]$, Rellich compactness gives a subsequence converging in $L^2(-R,R)$.  A diagonal argument with the uniform tail estimate gives the convergence in $L^2(\R)$ along a subsequence. Hence, the resolvent of $H_\infty$ is compact.

By the Rayleigh principle, the lowest eigenvalue $\lambda_\infty$ admits a minimizer.  Since replacing a minimizer by its absolute value does not change the form, there exists a nonnegative minimizer $\psi_\infty\ge 0$.  The eigenvalue equation 
\[
        -\psi_\infty'' + (|x| - \lambda_\infty)\psi_\infty = 0 
\]
combined with $|x| \in C^{0,1}_{\mathrm{loc}}(\R)$ implies that  $\psi_\infty \in C^{2,1}_{\mathrm{loc}}(\R)$. If $\psi_\infty(x_0) = 0$ for some $x_0 \in \R$, then $\psi_\infty'(x_0) = 0$ since $x_0$ is a global minimum. By the uniqueness theorem for ordinary differential equations, this initial value would force $\psi_\infty \equiv 0$, contradicting $\|\psi_\infty\|_{L^2(\R)} > 0$. Hence, the ground state is strictly positive, which in turn implies the simplicity of
the lowest eigenvalue. In particular, the ground state is unique up to scaling.

Since the potential $|x|$ is even, the positive ground state
$\psi_\infty$ must also be even, so that $\psi_\infty'(0) = 0$.  On the
positive half-line $x > 0$, the equation
$-\psi'' + x\psi = \lambda_\infty \psi$ reduces to the Airy equation
$-y''(t) + t y(t) = 0$ under the shift $t = x - \lambda_\infty$, whose general solution is
\[
        \psi(x) = a\operatorname{Ai}(x-\lambda_\infty) + b\operatorname{Bi}(x-\lambda_\infty).
\]
The requirement $\psi_\infty \in L^2(\R)$ forces $b = 0$, since
$\operatorname{Bi}$ grows exponentially at $+\infty$.  The boundary
condition $\psi_\infty'(0) = 0$ then implies
$\operatorname{Ai}'(-\lambda_\infty) = 0$, while the normalization
$\psi_\infty(0) = 1$ gives
\begin{equation*}
        \psi_\infty(x)=
        \frac{\operatorname{Ai}(|x|-\lambda_\infty)}
             {\operatorname{Ai}(-\lambda_\infty)},
\end{equation*}
which completes the proof.
\end{proof}
\begin{remark}
The strict positivity of $\psi_\infty$ on the entire real line can also be verified directly from the explicit formula \eqref{eq:Airy-ground} by comparing the zeros of the Airy function and its derivative. Let $a_1 \approx -2.338$ and $a'_1 \approx -1.019$ denote the first (largest negative) zeros of $\operatorname{Ai}(z)$ and $\operatorname{Ai}'(z)$, respectively. 
By Theorem \ref{thm:Airy-full}, we have $-\lambda_\infty = a'_1$. Therefore, for any $x \in \R$, 
\[
\av{x}-\lambda_\infty\ge a'_1 >a_1.
\]
which implies $\operatorname{Ai}(\av{x}-\lambda_\infty)>0$, i.e., $\psi_\infty(x) > 0$ for all $x \in \R$.
\end{remark}

\subsection{Dirichlet approximation of the Airy operator}\label{subsec:Airy-approx}

We now approximate the full-line Airy ground state by Dirichlet eigenfunctions on expanding intervals.
For $L>0$, consider the operator $-d^2/dx^2+|x|$ on $(-L,L)$. Let $\lambda_L$ be its first Dirichlet eigenvalue and let $\psi_L>0$ be its first Dirichlet eigenfunction, normalized by $\psi_L(0)=1$.

\begin{theorem}[Monotone convergence of Airy eigenpairs]\label{thm:Airy-monotone}
As $L\to\infty$,
\[
        \lambda_L\downarrow\lambda_\infty.
\]
The first eigenfunctions are even, and after extension by zero outside $(-L,L)$,
\begin{equation}\label{eq:Airy-monotone}
        0\le\psi_{L_1}(x)\le\psi_{L_2}(x)\le\psi_\infty(x),
        \qquad 0<L_1<L_2,
\end{equation}
for every $x\in\R$.  Moreover, $\psi_L\uparrow\psi_\infty$ pointwise and locally in $C^2(\R)$.
\end{theorem}

\begin{proof}
The Rayleigh quotient is monotone in the domain, so $\lambda_{L_2}<\lambda_{L_1}$ for $L_2>L_1$, where the strictness follows from the strong maximum principle. The Rayleigh principle also yields $\lambda_\infty<\lambda_L$. Conversely, testing with compactly supported smooth functions whose Rayleigh quotient is $\lambda_\infty+\varepsilon$, we have $\limsup_{L\to\infty}\lambda_L\le\lambda_\infty+\varepsilon$. Thus, $\lambda_L\downarrow\lambda_\infty$.

Since $\psi_L\ge 0$ is even and $\psi_L(0)=1$, it suffices to show $\psi_{L_1}(x)<\psi_{L_2}(x)$ on $(0,L_2)$ for any $L_1<L_2$. Consider the function 
\[
\Psi(x) = \frac{\psi_{L_1}(x)}{\psi_{L_2}(x)}.
\]
Then
\[
\Psi^\prime(x) = \frac{\psi_{L_1}^\prime(x)\psi_{L_2}(x)-\psi_{L_1}(x)\psi_{L_2}^\prime(x)}{\psi_{L_2}(x)^2},
\]
and $\Psi(0) = 1$, $\Psi^\prime(0) = 0$ and $\Psi(x) = 0$ for $L_1\le x<L_2$. Since 
\[
(\psi_{L_1}^\prime(x)\psi_{L_2}(x)-\psi_{L_1}(x)\psi_{L_2}^\prime(x))' = \bracket{\lambda_{L_2}-\lambda_{L_1}}\psi_{L_1}(x)\psi_{L_2}(x)<0,
\]
it follows that $\Psi^\prime(x) < 0$ and $\Psi(x) < \Psi(0) = 1$ for $0< x\le L_1$. Hence, $\psi_{L_1}(x)<\psi_{L_2}(x)$ on $(0,L_2)$ for any $L_1<L_2$, which gives the monotonicity in \eqref{eq:Airy-monotone}. An analogous argument, using $\lambda_\infty<\lambda_L$, gives $\psi_{L}(x)\le\psi_\infty(x)$. Finally, together with the equation $-\psi_L''+|x|\psi_L=\lambda_L\psi_L$, the monotonicity~\eqref{eq:Airy-monotone} implies local $C^2$ convergence of $\psi_L$ to $\psi_\infty$.
\end{proof}

\subsection{A convex Schr\"odinger potential with unscaled \texorpdfstring{$1/2$}{1/2} failure}\label{subsec:Airy-failure}

Now we are ready to show that for the convex potential $V(x)=|x|$, unscaled $1/2$-logconcavity fails on
sufficiently large intervals.
\begin{theorem}[The Airy potential $V(x)=|x|$ gives unscaled failure]\label{thm:Airy-failure}
For all sufficiently large $L$, the first Dirichlet eigenfunction $\psi_L$ of $-d^2/dx^2+|x|$ on $(-L,L)$, normalized by $\psi_L(0)=1$, has the property that
\[
        x\longmapsto\sqrt{-\log\psi_L(x)}
\]
is not convex. 
\end{theorem}

\begin{proof}
For $x\ge0$, the full-line ground state is
\[
        \psi_\infty(x)=
        \frac{\operatorname{Ai}(x-\lambda_\infty)}
             {\operatorname{Ai}(-\lambda_\infty)}.
\]
Let $v_\infty=-\log\psi_\infty\ge 0$.  The standard Airy asymptotics give, as $x\to+\infty$,
\[
        v_\infty(x)=\frac23(x-\lambda_\infty)^{3/2}+\frac14\log (x-\lambda_\infty)+O(1),
\]
\[
        v_\infty'(x)=(x-\lambda_\infty)^{1/2}+O(x^{-1}),
        \qquad
        v_\infty''(x)=\frac12(x-\lambda_\infty)^{-1/2}+O(x^{-2}).
\]
Consequently,
\begin{equation}\label{eq:QAiry-negative}
        Q_\infty(x):=2v_\infty(x)v_\infty''(x)-
        (v_\infty'(x))^2
        =-\frac13(x-\lambda_\infty)+o(x)<0
\end{equation}
for all sufficiently large $x$.  Choose $x_*>0$ such that $Q_\infty(x_*)<0$.  By Theorem \ref{thm:Airy-monotone}, $\psi_L\to\psi_\infty$ locally in $C^2$, so $v_L=-\log\psi_L$ converges to $v_\infty$ in $C^2$ near $x_*$.  Hence, for large $L$,
\[
        Q_L(x_*):=2v_L(x_*)v_L''(x_*)-(v_L'(x_*))^2<0,
\]
which implies $(\sqrt{v_L})''(x_*)<0$. Thus, $\sqrt{-\log\psi_L}$ is not convex.
\end{proof}

\begin{corollary}[Smooth convex counterexamples]\label{cor:smooth-counter}
There exist $L>0$ and a convex potential $V\in C^\infty([-L,L])$ such that the first Dirichlet ground state $\psi_V$ on $(-L,L)$ is not unscaled $1/2$-logconcave.
\end{corollary}

\begin{proof}
Choose $L$ as in Theorem \ref{thm:Airy-failure} and replace $|x|$ by
\[
        V_\varepsilon(x)=\sqrt{x^2+\varepsilon^2}.
\]
Then $V_\varepsilon$ is smooth and convex, and $V_\varepsilon\to |x|$ uniformly on $[-L,L]$ as $\varepsilon\to 0$. Since the first eigenvalue is simple, the normalized eigenfunctions converge in $C^2_{\rm loc}(-L,L)$.  So for all sufficiently small $\varepsilon>0$, $\sqrt{-\log\psi_{V_\varepsilon}}$ is not convex.
\end{proof}

\subsection{An affine log-concave weight: unscaled failure and scaled recovery}\label{subsec:affine-example}

Our second class of counterexamples concerns weighted Laplacians
with affine log-concave weights.  Here, unscaled $1/2$-logconcavity also
fails, but the scaled version is restored under an explicit condition
on the constant $a$.

\begin{theorem}[Affine log-concave weight: unscaled failure]\label{thm:affine-failure}
Let $I=(0,\pi)$, let $c\in\R\setminus\{0\}$, and consider
\[
        d\mu_c=e^{-cx}\dd x,
        \qquad
        \Delta_c=\frac{d^2}{dx^2}-c\frac d{dx}.
\]
The positive first Dirichlet eigenfunction is
\begin{equation}\label{eq:affine-eigenfunction}
        \psi_c(x)=A_ce^{cx/2}\sin x,
\end{equation}
where $A_c$ is chosen so that $\max_I\psi_c=1$.  The function $\sqrt{-\log \psi_c}$ is not convex on $I$.
\end{theorem}

\begin{proof}
A direct computation gives
\[
        -\Delta_c(e^{cx/2}\sin x)
        =\left(1+\frac{c^2}{4}\right)e^{cx/2}\sin x,
\]
so \eqref{eq:affine-eigenfunction} gives the first Dirichlet eigenfunction.  Setting $v_c=-\log \psi_c$, we obtain
\[
        v_c'(x)=-\frac c2-\cot x,
        \qquad
        v_c''(x)=\csc^2x,
        \qquad
        v_c'''(x)=-2\csc^2x\cot x.
\]
The unique maximum point $x_c$ of $\psi_c$ satisfies $v_c'(x_c)=0$, which gives $\cot x_c=-c/2$. Hence,
\[
        A:=v_c''(x_c)=1+\frac{c^2}{4}>0,
        \qquad
        B:=v_c'''(x_c)=c\left(1+\frac{c^2}{4}\right)\ne0.
\]
Since $\psi_c(x_c)=1$ yields $v_c(x_c)=0$, the Taylor expansion of
$v_c$ at $x_c$ with $y=x-x_c$ gives,
\[
        v_c(x_c+y)=\frac A2y^2+\frac B6y^3+O(y^4).
\]
Consequently,
\[
        \sqrt{v_c(x_c+y)}=
        \begin{cases}
        \sqrt{A/2}\,y+\dfrac{B}{6\sqrt{2A}}y^2+O(y^3),& y>0,\\[1.2ex]
        -\sqrt{A/2}\,y-\dfrac{B}{6\sqrt{2A}}y^2+O(|y|^3),& y<0.
        \end{cases}
\]
The one-sided second derivatives at the maximum are
\[
        (\sqrt{v_c})''(x_c^+)=\frac{B}{3\sqrt{2A}},
        \qquad
        (\sqrt{v_c})''(x_c^-)=-\frac{B}{3\sqrt{2A}}.
\]
One of them is negative because $B\ne0$.  Hence, $\sqrt{-\log \psi_c}$ cannot be convex.
\end{proof}

\begin{proposition}[The same affine example is scaled $1/2$-logconcave]\label{prop:affine-scaled}
For the eigenfunction $\psi_c$ in Theorem \ref{thm:affine-failure}, the scaled $1/2$-logconcavity property does hold.  More precisely, if
\[
        a\ge \frac12\bracket{1+\frac{c^2}{4}},
\]
then $\sqrt{a-\log \psi_c}$ is convex on $(0,\pi)$.
\end{proposition}

\begin{proof}
Using the one-dimensional criterion \eqref{eq:half-criterion-1d}, it is enough to check
\[
        2(a+v_c)v_c''-(v_c')^2\ge0.
\]
Since $v_c\ge0$ and $v_c''=\csc^2x$, we have
\[
        \frac{(v_c')^2}{v_c''}
        =\left(-\frac c2\sin x-\cos x\right)^2
        \le 1+\frac{c^2}{4}.
\]
The lower bound on $a$ yields
\[
        2(a+v_c)v_c''-(v_c')^2
        \ge 2av_c''-(v_c')^2\ge0,
\]
which completes the proof.
\end{proof}

\section*{Acknowledgments}
 We would like to thank Professors Andrea Colesanti and Paolo Salani for their suggestions and warm encouragement. The  third author is supported by NSF of Jiangsu Province No. BK20231309.

\section*{Declarations}

\noindent\textbf{Conflict of Interest}\ \ The authors declare that they have no conflict of interest.

\vspace{1em}
\noindent\textbf{Data Availability}\ \ Data sharing is not applicable to this article as no datasets were generated or analyzed during the current study.

\bibliographystyle{amsalpha}
\bibliography{references}

\end{document}